\documentclass{amsart}
\usepackage{amssymb}
\usepackage{amsfonts}

\newtheorem{theorem}{Theorem}
\theoremstyle{plain}

\renewcommand\bigskip\medskip

\begin{document}
\title[]{Hyperquadratic continued fractions \\ and automatic sequences}
\author{Alain LASJAUNIAS and Jia-Yan YAO}
\date{\today}

\begin{abstract}
The aim of this note is to show the existence of a correspondance between
certain algebraic continued fractions in fields of power series over a
finite field and automatic sequences in the same finite field. This
connection is illustrated by three families of examples and a counterexample.
\end{abstract}

\subjclass{Primary 11J70, 11T55; Secondary 11B85}
\keywords{finite fields,
power series over a finite field, continued fractions, finite automata,
automatic sequences}
\maketitle

\section{Introduction}

Let $\mathbb{F}_{q}$ be the finite field containing $q$ elements, with $%
q=p^{s}$ where $p$ is a prime number and $s\geqslant 1$ is an integer. We
consider the field of power series in $1/T$, with coefficients in $\mathbb{F}_{q}$,
 where $T$ is a formal indeterminate. We will denote this field by $%
\mathbb{F}(q)$. Hence a non-zero element of $\mathbb{F}(q)$ is written as $%
\alpha =\sum_{k\leqslant k_{0}}a_{k}T^{k}$ with $k_{0}\in \mathbb{Z}$, $%
a_{k}\in \mathbb{F}_{q}$, and $a_{k_{0}}\neq 0$. Noting the analogy of this
expansion with a decimal expansion for a real number, it is natural to
regard the elements of $\mathbb{F}(q)$ as (formal) numbers and indeed they
are analogue to real numbers in many ways.

It is well known that the sequence of coefficients (or digits) $%
(a_{k})_{k\leqslant k_{0}}$ for $\alpha $ is ultimately periodic if and only
if $\alpha $ is rational, that is $\alpha $ belongs to $\mathbb{F}_{q}(T)$.
However, and this is a singularity of the formal case, this sequence of
digits can also be characterized for all elements in $\mathbb{F}(q)$ which
are algebraic over $\mathbb{F}_{q}(T)$. The origin of the following theorem
can be found in the work of Christol \cite{C} (see also the article of Christol, Kamae, Mend\`{e}s France, and Rauzy \cite {CKMFR}).

\begin{theorem}[Christol]
Let $\alpha $ in $\mathbb{F}(q)$ with $q=p^s$. Let $(a_{k})_{k\leqslant k_{0}}$ be the sequence of
digits of $\alpha$ and $u(n)=a_{-n}$ for all integers $n\geqslant 0$. Then $\alpha $ is
algebraic over $\mathbb{F}_{q}(T)$ if and only if the following set of
subsequences of $(u(n))_{n\geqslant 0}$
\begin{equation*}
K(u)=\left\{ {(u(p^{i}n+j))_{n\geqslant 0}\,|\,\,i\geqslant
0,\,0\leqslant j<p}^{i}\right\}
\end{equation*}%
is finite.
\end{theorem}

The sequences having the finiteness property stated in this theorem were
first introduced in the 1960's by computer scientists. Considered in a
larger setting (see the beginning of Section 3), they are now called
automatic sequences, and form a class of deterministic sequences which can
be defined in several different ways. A full account on this topic and a
very complete list of references are to be found in the book of Allouche and Shallit \cite{AS}. In this note we want to show a different type
of connection between automatic sequences and some particular algebraic
power series in $\mathbb{F}(q)$.

Firstly, let us describe these particular algebraic elements. Let $\alpha $
be irrational in $\mathbb{F}(q)$. We say that $\alpha $ is hyperquadratic,
if there exists $r=p^{t}$, where $t\geqslant 0$ is an integer, such that the
elements $\alpha ^{r+1}$, $\alpha ^{r}$, $\alpha $, and $1$ are linked over $%
\mathbb{F}_{q}(T)$. Hence an hyperquadratic element is algebraic over $%
\mathbb{F}_{q}(T)$, of degree $\leqslant r+1$. The subset of hyperquadratic
elements in $\mathbb{F}(q)$ is denoted $\mathcal{H}(q)$. Note that this
subset contains the quadratic power series (take $r=1$) and also the cubic
power series (take $r=p$). Originally, these algebraic elements were
introduced in the 1970's by Baum and Sweet (see \cite{BS1}), in the
particular case $q=2$, and later considered in the 1980's by Mills and
Robbins \cite{MR} and Voloch \cite{V}, in all characteristic. It appears that $\mathcal{H}%
(q)$ contains elements having an arbitrary large algebraic degree. But
hyperquadratic power series are rare: an algebraic power series of high
algebraic degree has a small probability to be hyperquadratic. For different reasons,
this subset $\mathcal{H}(q)$ could be regarded as the analogue of the subset
of quadratic real numbers.

Besides, it is well known that any irrational element $\alpha $ in $\mathbb{F%
}(q)$ can be expanded as an infinite continued fraction where the partial
quotients $a_{n}$ are polynomials in $\mathbb{F}_{p}[T]$, all of positive
degree, except perhaps for the first one. We will use the traditional
notation $\alpha =[a_{1},a_{2},\ldots ,a_{n},\ldots ]$. The explicit
description of continued fractions for algebraic power series over a finite
field goes back to the works \cite{BS1,BS2} of Baum and Sweet, again when
the base field is $\mathbb{F}_{2}$. It was carried on ten years later by
Mills and Robbins in \cite{MR}. In the real case, no explicit continued
fraction expansion, algebraic of degree $n>2$, is known. On the other hand this expansion
for quadratic real numbers is well known to be ultimately periodic. In the formal
case, the situation is more complex. Quadratic power series have also an
ultimately periodic continued fraction expansion, but many other
hyperquadratic continued fractions can also be explicitly described. Most of
the elements in $\mathcal{H}(q)$ have an unbounded sequence of partial
quotients, but there are also expansions with all partial quotients of
degree $1$. This last phenomenon was discovered firstly by Mills and Robbins
in \cite{MR}, and later studied more deeply by Lasjaunias and Yao in \cite%
{LY}. Even though the pattern of hyperquadratic expansions can sometimes be very
sophisticated (see for instance the work \cite{F} of Firicel, where a
generalization of the cubic introduced by Baum and Sweet is presented), it
is yet doubtful wether this description, even partial, is possible for all
power series in $\mathcal{H}(q)$.

Power series in $\mathcal{H}(q)$ have particular properties concerning
Diophantine approximation and this is also why they were first considered.
The work \cite{M} of Mahler in this area, is fundamental. There, a first
historical example of hyperquadratic power series, on which we come back
below in this note, was introduced. Note that the irrationality measure (also called approximation exponent, see for instance \cite[p.214]{L2}) of a
power series can be computed if the explicit continued fraction for this
element is known. In this way, for many elements in $\mathcal{H}(q)$, the
irrationality measure, often greater than $2$ for non-quadratic elements,
has been given. Hence, contrary to the real case, many algebraic power
series of degree $>2$, most of them hyperquadratic, are known to have an
irrationality measure greater than $2.$ Actually, for algebraic power series
which are not hyperquadratic, concerning their continued fraction expansions
and their irrationality measure, not so much is known. The reader may
consult Schmidt \cite{S} and Lasjaunias \cite{L2}, for instance, for more
informations and references on this matter.

With each infinite continued fraction in  $\mathbb{F}(q)$, we can associate a sequence
in $\mathbb{F}_{q}^*$ in the following way: if $\alpha =[a_{1},a_{2},\ldots ,a_{n},\ldots ]$
then, for $n\geq 1$, we define $u(n)$ as the leading coefficient of the polynomial $a_n$.
For several examples in $\mathcal{H}(q)$, we
have observed that the sequence $(u(n))_{n\geq 1}$ is automatic.
Indeed, a first observation in this area is due to Allouche \cite{A}. In the
article \cite{MR} of Mills and Robbins, for all $p\geqslant 5$, a particular family of continued
fractions in $\mathcal{H}(p)$, having $a_{n}=\lambda _{n}T$ for all integers $n\geqslant 1$ with
$\lambda _{n}$ in $\mathbb{F}_{p}^{\ast }$, was introduced. Shortly after the publication of
\cite{MR}, Allouche could prove in \cite{A} that the sequence $(\lambda
_{n})_{n\geqslant 1}$ in $\mathbb{F}_{p}^{\ast }$ is automatic (see also the
last section of \cite{LY}, where this question is discussed in a larger
context). In the present note (Section 2), we shall describe several
families of hyperquadratic continued fractions and we show in Section 3 that
the sequence associated with them, as indicated above, is also automatic.

Yet it is an open question to know wether this is true for all elements in $%
\mathcal{H}(q)$. If the answer were negative, it would be interesting to be
able to characterize the elements in $\mathcal{H}(q)$ which have this
property. As mentioned above, very little is known, concerning continued
fractions, for algebraic power series which are not hyperquadratic. However
an element in $\mathbb{F}(3)$, algebraic of degree $4$, was introduced by
Robbins and Mills in \cite{MR}. This element is not hyperquadratic. In this note (Section 4),
we show that the sequence in $\mathbb{F}_{3}^{\ast }$, associated as above with its
 continued fraction expansion, is not automatic.

\section{Three families of hyperquadratic continued fractions}

In this section we shall use the notation and results found in \cite{L3}.

Let $\alpha $ be an irrational element in $\mathbb{F}(q)$ with $\alpha
=[a_{1},\ldots ,a_{n},\ldots ]$ as its continued fraction expansion. We
denote by $\mathbb{F}(q)^{+}$ the subset of $\mathbb{F}(q)$ containing the
elements having an integral part of positive degree (i.e. with $\deg
(a_{1})>0$). For all integers $n\geqslant 1$, we put $\alpha
_{n}=[a_{n},a_{n+1},\ldots ]$ ($\alpha_1=\alpha$), and we introduce the continuants
 $x_{n},y_{n}\in \mathbb{F}_{q}[T]$ such that $x_{n}/y_{n}=[a_{1},a_{2},\ldots
,a_{n}] $. As usual we
extend the latter notation to $n=0$ with $x_{0}=1$ and $y_{0}=0$. Observe
that the notation used here for the continuants $x_{n}$ and $y_{n}$ is
different from the one used in \cite{L3}, and hopefully simplified.

As above we set $r=p^{t}$, where $t\geqslant 0$ is an integer. Let $P,Q\in
\mathbb{F}_{q}[T]$ such that $\deg (Q)<\deg (P)<r$. Let $\ell \geqslant 1$
be an integer and $A_{\ell }=(a_{1},a_{2},\ldots ,a_{\ell })$ a vector in $(%
\mathbb{F}_{q}[T])^{\ell }$ such that $\deg (a_{i})>0$ for $1\leqslant
i\leqslant \ell $. Then by Theorem 1 in \cite{L3}, there exists an infinite
continued fraction in $\mathbb{F}(q)$ defined by $\alpha
=[a_{1},a_{2},\ldots ,a_{\ell },\alpha _{\ell +1}]$ such that $\alpha
^{r}=P\alpha _{\ell +1}+Q$. This element $\alpha $ is hyperquadratic and it
is the unique root in $\mathbb{F}(q)^{+}$ of the following algebraic
equation:
\begin{equation}
y_{\ell }X^{r+1}-x_{\ell }X^{r}+(y_{\ell -1}P-y_{\ell }Q)X+x_{\ell
}Q-x_{\ell -1}P=0.  \label{eq0}
\end{equation}%
Note that if $r=1$, then $\alpha $ is quadratic. In this case $P$
is a nonzero constant polynomial, i.e. $P=\varepsilon \in \mathbb{F}%
_{q}^{\ast }$ and $Q=0$. Given $\ell $ and $A_{\ell }$, we have $\alpha
=\varepsilon \alpha _{\ell +1}$, and this implies $a_{\ell +m}=\varepsilon
^{(-1)^{m}}a_{m}$, for all integers $m\geqslant 1$. Hence the continued
fraction expansion is purely periodic (a simple computation shows that the
length of the period is at most $2\ell $ if $\ell $ is odd or $(q-1)\ell $
if $\ell $ is even).
\medskip

We shall describe three families of
continued fractions generated as above.
\medskip

\textbf{First family: }$\mathcal{F}1$. The simplest and first case that we
consider is $(P,Q)=(\varepsilon ,0)$ where $\varepsilon \in \mathbb{F}%
_{q}^{\ast }$ and consequently $\alpha^r=\varepsilon \alpha_{l+1}$. Due to the Frobenius isomorphism, we obtain, in the same way
as above for $r=1$, the relation $a_{\ell +m}=\varepsilon
^{(-1)^{m}}a_{m}^{r}$ for all integers $m\geqslant 1$. Hence the continued
fraction, depending on the arbitrary given $\ell $ first partial quotients,
is fully explicit. These hyperquadratic continued fractions were studied
independently by Schmidt \cite{S} and Thakur \cite{T} (particularly for $%
\varepsilon =1$). Let us recall that the elements in $\mathcal{H}(q)$,
called here hyperquadratic, were first named in \cite{L1} as algebraic of
class I, and then the elements studied by Schmidt and Thakur were called
algebraic of class IA. The possibility of choosing arbitrarily the vector $%
A_{\ell }$ has an important consequence. Even though this is not truly the
matter of this note, we have already mentioned the irrationality measure $%
\nu (\alpha )$ of $\alpha $ in $\mathbb{F}(q)$. In his fundamental work \cite%
{M}, Mahler established (following an old result of Liouville in the real
case) that if $\alpha \in \mathbb{F}(q)$ is algebraic of degree $d\geqslant 2$ over $%
\mathbb{F}_{q}(T)$, then we have $\nu (\alpha ) \in [2,d]$. Besides $\nu
(\alpha )$ is directly depending on the sequence of the degrees of the
partial quotients for $\alpha $ (see for instance \cite[p.214]{L2}). For an
element of $\mathcal{F}1$, this sequence of degrees $(d_{n})_{n\geqslant 1}$%
, satisfies $d_{\ell +m}=rd_{m}$ for all integers $m\geqslant 1$, and
consequently $d_{n}$ depends directly on the first $\ell $ degrees. Hence,
Schmidt and Thakur, independently, could obtain (by a sophisticated
computation) the irrationality measure for such an element, depending on $r$
 and the first $\ell $ degrees. In this way they could establish the
following result: for each rational number $\mu $ in the range $[2,+\infty [$%
, there exists $\alpha $ in $\mathcal{F}1$ such that $\nu (\alpha )=\mu $.

Let us make a last observation on the simplest element in $\mathcal{F}1$. We
take $\ell =1$ and $a_{1}=T$, with $(P,Q)=(1,0)$, then the corresponding
continued fraction is $\Theta _{1}=[T,T^{r},T^{r^{2}},\ldots
,T^{r^{n}},\ldots ]$, and $\Theta _{1}$ satisfies $X=T+1/X^{r}$. Here the
irrationality measure is easy to compute, and indeed we have $\nu (\Theta
_{1})=r+1$. Since the degree $d$ of $\Theta _{1}$ satisfies $d\leqslant r+1$%
, using Mahler's argument, we see that $\Theta _{1}$ has algebraic degree
equal to $r+1$. Note that for a general element $\alpha $ in $\mathcal{F}1$,
its exact algebraic degree is an undecided question.
\medskip

\textbf{Second and third families: }$\mathcal{F}2$ and $\mathcal{F}3$.%
\textbf{\ }For an element $\alpha $ in $\mathcal{F}2$, we assume $%
(P,Q)=(\varepsilon _{1}T,\varepsilon _{2})$; and for an element $\alpha $ in $%
\mathcal{F}3$, we assume $(P,Q)=(\varepsilon _{1}T^{2},\varepsilon _{2}T)$,
where $(\varepsilon _{1},\varepsilon _{2})$ is a pair in $(\mathbb{F}%
_{q}^{\ast })^{2}$. Note that here we have $r>1$ for elements in $\mathcal{F}%
2$, but $r>2$ for elements in $\mathcal{F}3$. Our aim is to give the
explicit continued fraction expansion for $\alpha $. The integer $\ell\geqslant 1$, as
above, is chosen arbitrarily. However we need to impose a
restriction on the choice of the vector $A_{\ell }$, in both cases: we
assume that $T$ divides $a_{i}$ for all integers $i$ with $1\leqslant
i\leqslant \ell $. Then we can describe the sequence of partial quotients in
the continued fraction expansion in both cases. Given $A_{\ell }$, chosen as
indicated, for all integers $n\geqslant 0$, we have: if $\alpha $ is in $%
\mathcal{F}2$
\begin{equation*}
a_{\ell +4n+1}=\frac{a_{2n+1}^{r}}{\varepsilon _{1}T},\ a_{\ell +4n+2}=-%
\frac{\varepsilon _{1}}{\varepsilon _{2}}T,\ a_{\ell +4n+3}=-\frac{%
\varepsilon _{2}^{2}a_{2n+2}^{r}}{\varepsilon _{1}T},\ a_{\ell +4n+4}=\frac{%
\varepsilon _{1}}{\varepsilon _{2}}T
\end{equation*}%
and if $\alpha $ is in $\mathcal{F}3$
\begin{equation*}
a_{\ell +4n+1}=\frac{a_{2n+1}^{r}}{\varepsilon _{1}T^{2}},\ a_{\ell +4n+2}=-%
\frac{\varepsilon _{1}}{\varepsilon _{2}}T,\ a_{\ell +4n+3}=-\frac{%
\varepsilon _{2}^{2}a_{2n+2}^{r}}{\varepsilon _{1}},\ a_{\ell +4n+4}=\frac{%
\varepsilon _{1}}{\varepsilon _{2}}T.
\end{equation*}%
The method to obtain these formulas is the same in both cases and it has
been explained in \cite{L3}. Actually the formulas for $\alpha $ in $%
\mathcal{F}2$ were published in \cite{L3} on the top of p.334. However, note
that there is a mistake in the final statement given there and the pair $%
(k,k+1)$ to the right hand side of the first and third formulas must be
replaced by the pair $(2k-1,2k)$. Concerning the formulas for $\alpha $ in $%
\mathcal{F}3$, they were given by Firicel in \cite{F} (case $\varepsilon
_{1}=\varepsilon _{2}=-1$). The reader is invited to consult this last work
where the method is clearly explained. Note that the particular choice of
the vector $A_{\ell }$, has been made in order to have an integer
(polynomial) when dividing by $T$ or $T^{2}$ (however a larger choice could
have been possible for elements in $\mathcal{F}3$, in a particular case: for
example, if $\ell $ is odd, then it would be enough to assume that $T$
divides $a_{i}$ only for odd indices $i$). Note that if we do not assume the
particular choice we made for the vector $A_{\ell }$, the continued fraction
for $\alpha $ does exist, but an explicit description is not given.

We make a comment concerning the last family $\mathcal{F}3$. It has been
introduced to cover a particular case of historical importance. In his
fundamental paper \cite{M}, Mahler presented the following power series $%
\Theta _{2}=1/T+1/T^{r}+\cdots +1/T^{r^{n}}+\cdots $. Note that $\Theta _{2}$
can be regarded as a dual of the element $\Theta _{1}$ in $\mathcal{F}1$,
presented above. However we clearly have $\Theta _{2}=1/T+\Theta _{2}^{r}$,
hence $\Theta _{2}$ is hyperquadratic and algebraic of degree $\leqslant r$.
Mahler observed that we have $\nu (\Theta _{2})=r$, and therefore the
algebraic degree of $\Theta _{2}$ is equal to $r$. Now let us consider the
element of $\mathcal{F}3$ defined as above by the pair $(P,Q)=(-T^{2},-T)$
with $\ell =1$ and $a_{1}=T$. Then $\alpha $ is the unique root in $\mathbb{F%
}(p)^{+}$ of the following algebraic equation:%
\begin{equation*}
X^{r+1}-TX^{r}+TX=0.
\end{equation*}%
Set $\beta =1/\alpha $, and we get $\beta =1/T+\beta ^{r}$. Since $1/\beta $
is in $\mathbb{F}(p)^{+}$, we obtain%
\begin{equation*}
\beta =\Theta _{2}=1/T+1/T^{r}+1/T^{r^{2}}+\cdots 
\end{equation*}%
Let us recall that the sequence of partial quotients for $\Theta _{2}$ has been long
known (see  for example \cite[p.633%
]{F} with the references therein. See also \cite[p.224]{L1} for a different approach).

As we wrote in the introduction, we are interested in the leading
coefficients of partial quotients. If $(a_{n})_{n\geqslant 1}$ is the
sequence of partial quotients for $\alpha $, we denote by $u(n)$ the leading
coefficient of $a_{n}$. If $\alpha $ is in $\mathcal{F}1$, since we have
$$
u(\ell +m)=\varepsilon^{(-1)^{m}}(u(m))^{r}
$$
for all integers $m\geqslant 1$, it is easy to see
that the sequence $(u(n))_{n\geqslant 1}$ is purely periodic. The length of
the period might be a multiple of $\ell $, and it depends on the value for $%
\varepsilon $ and the relationships between $p$, $q$ and $r$.

Now we turn to the case where $\alpha $ belongs to $\mathcal{F}2$ or $\mathcal{F}%
3 $. Both recursive definitions for the sequence of partial quotients,
wether $\alpha $ is in $\mathcal{F}2$ or $\mathcal{F}3$, imply the same
recursive definition for the corresponding sequence $(u(n))_{n\geqslant 1}$.
More precisely, for all integers $\ell \geqslant 1$, given $\ell $ values $%
u(1),u(2),\ldots ,u(\ell )$ in $\mathbb{F}_{q}^{\ast }$, we have, for all
integers $n\geqslant 0$,%
\begin{equation*}
\begin{array}{ll}
u(\ell +4n+1)=\varepsilon _{1}^{-1}(u(2n+1))^{r}, & u(\ell
+4n+2)=-\varepsilon _{1}\varepsilon _{2}^{-1},\  \\
u(\ell +4n+3)=-\varepsilon _{1}^{-1}\varepsilon _{2}^{2}(u(2n+2))^{r}, &
u(\ell +4n+4)=\varepsilon _{1}\varepsilon _{2}^{-1}.%
\end{array}%
\end{equation*}%
In the next section we shall see in Theorem \ref{thm2} that sequences of
this type, for all integers $\ell \geqslant 1$, are $2$-automatic sequences.

To conclude this section, we go back to the special element $1/\Theta _{2}$
in $\mathcal{F}3$, and the associated sequence $(u(n))_{n\geqslant 1}$ where
$l=1$ and $u(1)=1$. The latter is remarkable and has been studied extensively (see for
example \cite[Section 6.5]{AS}). We define
recursively the sequence of finite words $(W_{n})_{n\geqslant 1}$ by $%
W_{1}=1 $, and $W_{n+1}=W_{n},-1,W_{n}^{R}$, where
commas indicate here concatenation of words, and $W_{n}^{R}$ is the
reverse of the finite word $W_{n}$. Let $W$ be the infinite word beginning
with $W_{n}$ for all integers $n\geqslant 1$. Then one can check that the
sequence $(u(n))_{n\geqslant 1}$ coincides with $W$, and it is a special
paperfolding sequence which is known to be $2$-automatic (see for example
\cite[Theorem 6.5.4]{AS}).

\section{A family of automatic sequences}

In this section, we begin with the definition of automatic sequences. For
more details about this subject, see the book \cite{AS} of Allouche and
Shallit.

Let $A$ be a finite nonempty set, called an alphabet, of which every element
is called a letter. Fix $\emptyset $ an element not in $A$ and call it an
empty letter over $A$. Let $n\geqslant 0$ be an integer. If $n=0$, define $%
A^{0}=\{\emptyset \}$. For $n\geqslant 1$, denote by $A^{n}$ the set of all
finite sequences in $A$ of length $n$. Put $A^{\ast}=\bigcup_{n=0}^{+\infty }A^{n}$. Every element $w$ of $A^{\ast }$ is
called a word over $A$ and its length is noted $|w|$, i.e. $|w|=$ $n$ if $%
w\in A^{n}$.

Take $w,v\in A^{\ast }$. The concatenation between $w$ and $v$ (denoted by $%
w\ast v$ or more simply by $wv$) is the word over $A$ which begins with $w$
and is continued by $v$.

Now we give below a definition of finite automaton (see for example \cite{E}%
):\smallskip\

A finite automaton $\mathcal{A}=(S,s_{0},\Sigma ,\tau )$\ consists of

\begin{itemize}
\item an alphabet $S$ of states; one state $s_{0}$ is distinguished and
called initial state.

\item a mapping $\tau :S\times \Sigma \rightarrow S$, called transition
function, where $\Sigma $ is an alphabet containing at least two
elements.\smallskip
\end{itemize}

For any $a\in S$, put $\tau (a,\emptyset )=a$. Extend $\tau $ over $S\times
\Sigma ^{\ast }$ (noted again $\tau $) such that
\begin{equation*}
\forall \ a\in S\text{ and }l,m\in \Sigma ^{\ast },\text{ we have }\tau
(a,lm)=\tau (\tau (a,l),m).
\end{equation*}

Let $k\geqslant 2$ be an integer and $\Sigma _{k}=\{0,1,\ldots ,k-1\}$. We
call $v=(v(n))_{n\geqslant 0}$ a $k$-automatic sequence if there exist a
finite automaton $\mathcal{A}=(S,s_{0},\Sigma _{k},\tau )$ and a mapping $o$
defined on $S$ such that $v(0)=o(s_{0})$, and for all integers $n\geqslant 1$
with standard $k$-adic expansion $n=\sum_{j=0}^{l}n_{j}k^{j}$, we have $%
v(n)=o(\tau(s_{0},n_{l}\cdots n_{0}))$.

We recall that all ultimately periodic
sequences are $k$-automatic for all $k\geqslant 2$, adding a prefix to a sequence does not change
its automaticity, and that a sequence is $k$-automatic if and only if it is $%
k^{m}$-automatic for all integers $m\geqslant 1$. In this work, we consider sequences of the form $%
v=(v(n))_{n\geqslant 1}$, and we say that $v$ is $k$-automatic if the
sequence $(v(n))_{n\geqslant 0}$ is $k$-automatic, with $v(0)$ fixed
arbitrarily. We have the following important characterization:
a sequence $v=(v(n))_{n\geqslant 1}$ is $k$-automatic if and only if
its $k$-kernel
\begin{equation*}
K_{k}(v)=\left\{ {(v(k^{i}n+j))_{n\geqslant 1}\,|\,\,i\geqslant
0,\,0\leqslant j<k}^{i}\right\}
\end{equation*}%
is a finite set. The origin of this characterization for automatic sequences is due to S. Eilenberg \cite[p.107]{E},
who was one of the first to publish a general treatise on this matter.

Let $v=(v(n))_{n\geqslant 1}$ be a sequence. For all integers $n\geqslant 1$%
, we define%
\begin{equation*}
(T_{0}v)(n)=v(2n)\text{ and }(T_{1}v)(n)=v(2n+1).
\end{equation*}%
Then for all integers $n,a\geqslant 1$, and $0\leqslant b<2^{a}$ with binary
expansion%
\begin{equation*}
b=\sum_{j=0}^{a-1}b_{j}2^{j}\qquad (0\leqslant b_{j}<2),
\end{equation*}%
with the help of the operators $T_{0}$ and $T_{1}$, we obtain
\begin{equation*}
v(2^{a}n+b)=(T_{b_{a-1}}\circ T_{b_{a-2}}\circ \cdots \circ T_{b_{0}}v)(n).
\end{equation*}%
In particular, we obtain that $v$ is $2$-automatic if and only if both $%
T_{0}v,T_{1}v$ are $2$-automatic, for we have $K_{2}(v)=\{v\}\cup
K_{2}(T_{0}v)\cup K_{2}(T_{1}v)$.

With these definitions, we have the following theorem, which can be compared
with a result of Allouche and Shallit (see \cite[Theorem 2.2]{as2}).

\begin{theorem}
\label{thm3}Let $m\geqslant 0$ be an integer, $v=(v(n))_{n\geqslant 1}$ a
sequence in an alphabet $A$, and $\sigma $ a bijection on $A$.

\begin{enumerate}
\item[($0_{m}$)] If $T_{0}v$ is $2$-automatic, and $(T_{1}v)(n+m)=\sigma
(v(n)) $ for all integers $n\geqslant 1$, then $v$ is $2$-automatic;

\item[($1_{m}$)] If $T_{1}v$ is $2$-automatic, and $(T_{0}v)(n+m)=\sigma
(v(n)) $ for all integers $n\geqslant 1$, then $v$ is $2$-automatic.
\end{enumerate}
\end{theorem}

\begin{proof}
 Since $A$ is finite, there exists an integer $l\geqslant 1$ such that $\sigma ^{l}=\mathrm{id%
}_{A}$, the identity mapping on $A$.

In the following we shall show ($0_{m}$) and ($1_{m}$) by induction on $m$.

\smallskip If $m=0$, then under the condition of ($0_{0}$), we have $T_{1}v=\sigma(v)$ and
\begin{equation*}
K_{2}(v)=\{\sigma ^{j}(v)\,|\,0\leqslant j<l\}\cup
\bigcup_{j=0}^{l-1}\sigma ^{j}(K_{2}(T_{0}v)).
\end{equation*}%
Thus $K_{2}(v)$ is finite since $T_{0}v$ is $2$-automatic.

Similarly, under the condition in ($1_{0}$), we have $T_{0}v=\sigma(v)$ and%
\begin{equation*}
K_{2}(v)=\{\sigma ^{j}(v)\,|\,0\leqslant j<l\}\cup
\bigcup_{j=0}^{l-1}\sigma ^{j}(K_{2}(T_{1}v)).
\end{equation*}%
Thus $K_{2}(v)$ is finite since $T_{1}v$ is $2$-automatic.
\smallskip

If $m=1$, then we distinguish two cases below:

\textbf{Case (}$0_{1}$\textbf{):} $T_{0}v$ is $2$-automatic, and $%
(T_{1}v)(n+1)=\sigma (v(n))$ for all integers $n\geqslant 1$. Hence $%
T_{1}T_{1}v=\sigma (T_{0}v)$, and for all integers $n\geqslant 1$, we have
\begin{equation*}
(T_{0}T_{1}v)(n+1)=(T_{1}v)(2n+2)=\sigma (v(2n+1))=\sigma ((T_{1}v)(n)).
\end{equation*}%
Thus $T_{1}T_{0}T_{1}v=\sigma (T_{0}T_{1}v)$, and for all integers $%
n\geqslant 1$, we have
\begin{eqnarray*}
(T_{0}T_{0}T_{1}v)(n+1)&=&(T_{0}T_{1}v)(2n+2)=\sigma ((T_{1}v)(2n+1))\\
&=&\sigma^{2}(v(2n))=\sigma ^{2}((T_{0}v)(n)).
\end{eqnarray*}%
So $T_{0}T_{0}T_{1}v$ is $2$-automatic since it is obtained from $\sigma
^{2}(T_{0}v)$ by adding a letter before, and the latter is $2$-automatic,
for $T_{0}v$ is. Set $w=T_0T_1v$. Then $T_0w$ is $2$-automatic, and $T_1w=\sigma(w)$.
Thus by ($0_{0}$), we obtain that $T_0T_1v$ is $2$-automatic.
But $T_1T_1v=\sigma(T_0v)$ is also $2$-automatic, consequently $T_1v$ is $2$-automatic,
and then $v$ is $2$-automatic, for both $T_0v$ and $T_1v$ are $2$-automatic.

\textbf{Case (}$1_{1}$\textbf{):} $T_{1}v$ is $2$-automatic, and $%
(T_{0}v)(n+1)=\sigma (v(n))$ for all integers $n\geqslant 1$. Hence $%
T_{1}T_{0}v=\sigma (T_{0}v)$, and for all integers $n\geqslant 1$, we have
\begin{equation*}
(T_{0}T_{0}v)(n+1)=(T_{0}v)(2n+2)=\sigma (v(2n+1))=\sigma ((T_{1}v)(n)).
\end{equation*}%
So $T_{0}T_{0}v$ is $2$-automatic for it is obtained from $\sigma (T_{1}v)$
by adding a letter before. Set $w=T_{0}v$. Then $T_{1}w=\sigma (w)$ and $%
T_{0}w$ is $2$-automatic. Thus by ($0_{0}$), $w=T_{0}v$ is $2$-automatic,
and then $v$ is $2$-automatic since both $T_{0}v$ and $T_{1}v$ are $2$-automatic.
\smallskip

If $m=2$, then we distinguish again two cases below:

\textbf{Case (}$0_{2}$\textbf{):} $T_{0}v$ is $2$-automatic, and $%
(T_{1}v)(n+2)=\sigma (v(n))$ for all integers $n\geqslant 1$. Then for all
integers $n\geqslant 1$, we have
\begin{eqnarray*}
(T_{0}T_{1}v)(n+1) &=&(T_{1}v)(2n+2)=\sigma (v(2n))=\sigma ((T_{0}v)(n)), \\
(T_{1}T_{1}v)(n+1) &=&(T_{1}v)(2n+3)=\sigma (v(2n+1))=\sigma ((T_{1}v)(n)).
\end{eqnarray*}%
Hence $T_{0}T_{1}v$ is $2$-automatic for it is obtained from $\sigma
(T_{0}v) $ by adding a letter before. Put $w=T_{1}v$. Then $T_{0}w$ is $2$%
-automatic, and for all integers $n\geqslant 1$, we have $%
(T_{1}w)(n+1)=\sigma (w(n))$. Hence by ($0_{1}$), we obtain that $w=T_{1}v$
is $2$-automatic, and then $v$ is $2$-automatic since both $T_{0}v$ and $%
T_{1}v$ are $2$-automatic.

\textbf{Case (}$1_{2}$\textbf{):} $T_{1}v$ is $2$-automatic, and $%
(T_{0}v)(n+2)=\sigma (v(n))$ for all integers $n\geqslant 1$. Then for all integers
$n\geqslant 1$, we have
\begin{eqnarray*}
(T_{0}T_{0}v)(n+1) &=&(T_{0}v)(2n+2)=\sigma (v(2n))=\sigma ((T_{0}v)(n)), \\
(T_{1}T_{0}v)(n+1) &=&(T_{0}v)(2n+3)=\sigma (v(2n+1))=\sigma ((T_{1}v)(n)).
\end{eqnarray*}%
So $T_{1}T_{0}v$ is $2$-automatic for it is obtained from $\sigma (T_{1}v)$
by adding a letter before. Set $w=T_{0}v$. Then $T_{1}w$ is $2$-automatic,
and for all integers $n\geqslant 1$, we have $(T_{0}w)(n+1)=\sigma (w(n))$.
Hence by ($1_{1}$), we obtain that $w=T_{0}v$ is $2$-automatic, and then $v$
is $2$-automatic since both $T_{0}v$ and $T_{1}v$ are $2$-automatic.
\smallskip 

Now let $m\geqslant 2$ be an integer, and assume that both ($%
0_{j} $) and ($1_{j}$) hold for $0\leqslant j\leqslant m$. Then $[\frac{m}{2}%
]+1\leqslant m$. We shall show that both ($0_{m+1}$) and ($1_{m+1}$) hold.
For this, we distinguish two cases below:
\smallskip 

\textbf{Case (}$0_{m+1}$\textbf{):} $T_{0}v$ is $2$-automatic,
and $(T_{1}v)(n+m+1)=\sigma (v(n))$ for all integers $n\geqslant 0$. We
distinguish two cases again:

\textbf{Case 1:} $m$ is odd. Then for all integers $n\geqslant 1$, we have%
\begin{eqnarray*}
(T_{0}T_{1}v)(n+[\frac{m}{2}]+1) &=&(T_{1}v)(2n+m+1)=\sigma (v(2n))=\sigma
((T_{0}v)(n)), \\
(T_{1}T_{1}v)(n+[\frac{m}{2}]+1) &=&(T_{1}v)(2n+m+2)=\sigma (v(2n+1))=\sigma
((T_{1}v)(n)).
\end{eqnarray*}%
Hence $T_{0}T_{1}v$ is $2$-automatic for it is obtained from $\sigma
(T_{0}v) $ by adding a prefix of length $[\frac{m}{2}]+1$. Put $w=T_{1}v$.
Then $T_{0}w$ is $2$-automatic, and for all integers $n\geqslant 1$,%
\begin{equation*}
(T_{1}w)(n+[\frac{m}{2}]+1)=\sigma (w(n)).
\end{equation*}
By applying ($0_{[\frac{m}{2}]+1}$) with $w$, we obtain at once that $w$ is $2$%
-automatic, and then $v$ is $2$-automatic since both $T_{0}v$ and $T_{1}v$
are $2$-automatic.

\textbf{Case 2:} $m$ is even. Then for all integers $n\geqslant 1$, we have
\begin{eqnarray*}
(T_{0}T_{1}v)(n+[\frac{m}{2}]+1) &=&(T_{1}v)(2n+m+2)=\sigma (v(2n+1))=\sigma
((T_{1}v)(n)), \\
(T_{1}T_{1}v)(n+[\frac{m}{2}]) &=&(T_{1}v)(2n+m+1)=\sigma (v(2n))=\sigma
((T_{0}v)(n)).
\end{eqnarray*}%
So $T_{1}T_{1}v$ is $2$-automatic for it is obtained from $\sigma (T_{0}v)$
by adding a prefix of length $[\frac{m}{2}]$. Put $w=T_{1}v$. Then $T_{1}w$
is $2$-automatic, and $(T_{0}w)(n+[\frac{m}{2}]+1)=\sigma (w(n))$, for all
integers $n\geqslant 1$. By applying ($1_{[\frac{m}{2}]+1}$) with $w$, we
obtain that $w$ is $2$-automatic, and then $v$ is $2$-automatic since both $%
T_{0}v$ and $T_{1}v$ are $2$-automatic.

\smallskip \textbf{Case (}$1_{m+1}$\textbf{):} $T_{1}v$ is $2$-automatic,
and $(T_{0}v)(n+m+1)=\sigma (v(n))$ for all integers $n\geqslant 0$. We
distinguish two cases again:

\textbf{Case 1:} $m$ is odd. Then for all integers $n\geqslant 1$, we have%
\begin{eqnarray*}
(T_{0}T_{0}v)(n+[\frac{m}{2}]+1) &=&(T_{0}v)(2n+m+1)=\sigma (v(2n))=\sigma
((T_{0}v)(n)), \\
(T_{1}T_{0}v)(n+[\frac{m}{2}]+1) &=&(T_{0}v)(2n+m+2)=\sigma (v(2n+1))=\sigma
((T_{1}v)(n)).
\end{eqnarray*}%
Hence $T_{1}T_{0}v$ is $2$-automatic for it is obtained from $\sigma
(T_{1}v) $ by adding a prefix of length $[\frac{m}{2}]+1$. Put $w=T_{0}v$.
Then $T_{1}w$ is $2$-automatic, and for all integers $n\geqslant 1$,%
\begin{equation*}
(T_{0}w)(n+[\frac{m}{2}]+1)=\sigma (w(n)).
\end{equation*}%
By applying ($1_{[\frac{m}{2}]+1}$) with $w$, we obtain immediately that $w$ is $2$%
-automatic, and then $v$ is $2$-automatic since both $T_{0}v$ and $T_{1}v$
are $2$-automatic.

\textbf{Case 2:} $m$ is even. Then for all integers $n\geqslant 1$, we have
\begin{eqnarray*}
(T_{0}T_{0}v)(n+[\frac{m}{2}]+1) &=&(T_{0}v)(2n+m+2)=\sigma (v(2n+1))=\sigma
((T_{1}v)(n)), \\
(T_{1}T_{0}v)(n+[\frac{m}{2}]) &=&(T_{0}v)(2n+m+1)=\sigma (v(2n))=\sigma
((T_{0}v)(n)).
\end{eqnarray*}%
So $T_{0}T_{0}v$ is $2$-automatic for it is obtained from $\sigma (T_{1}v)$
by adding a prefix of length $[\frac{m}{2}]+1$. Put $w=T_{0}v$. Then $T_{0}w$
is $2$-automatic, and $(T_{1}w)(n+[\frac{m}{2}])=\sigma (w(n))$, for all
integers $n\geqslant 1$. By applying ($0_{[\frac{m}{2}]}$) with $w$, we
obtain that $w$ is $2$-automatic, and then $v$ is $2$-automatic since both $%
T_{0}v$ and $T_{1}v$ are $2$-automatic.

Finally we conclude that both ($0_{m}$) and ($1_{m}$) hold for all integers $%
m\geqslant 0$.
\end{proof}

We can now prove that the sequences, associated with the elements of $\mathcal{F}2$
and $\mathcal{F}3$ and described at the end of the previous section
are 2-automatic. This follows from the more general result stated below,
which is a direct application of Theorem~2.

\begin{theorem}
\label{thm2}Let $p$ be a prime number, $s\geqslant 1$ an integer, and $%
q=p^{s}$. Denote by $\mathbb{F}_{q}$ the finite field in $q$ elements. Set $%
r=p^{t}$, with $t\geqslant 0$ an integer. Fix $\alpha ,\beta ,\gamma ,\delta
\in \mathbb{F}_{q}^{\ast }$, and $\ell \geqslant 1$ an integer. Let $%
u=(u(n))_{n\geqslant 1}$ be a sequence in $\mathbb{F}_{q}^{\ast }$ such that for all
integers $n\geqslant 0$, we have%
\begin{equation}
\left\{
\begin{array}{l}
u(\ell +4n+1)=\alpha (u(2n+1))^{r},\text{ }u(\ell +4n+2)=\beta , \\
u(\ell +4n+3)=\gamma (u(2n+2))^{r},\text{ }u(\ell +4n+4)=\delta .%
\end{array}%
\right. \text{ }  \label{eq2}
\end{equation}%
Then the sequence $u$ is $2$-automatic.
\end{theorem}

\begin{proof}
For all $x,y\in \mathbb{F}_{q}^{\ast }$,
we put $\sigma_y (x)=yx^{r}$. Thence $\sigma_y$ is a bijection on $\mathbb{F}_{q}^{\ast }$.
For all integers $n\geqslant 1$, set $u_{0}(n)=u(2n)$ and $u_{1}(n)=u(2n+1)$,
and we need only show that both $u_{0}$ and $u_{1}$ are $2$-automatic. For
this, we distinguish below four cases.
\smallskip

\textbf{Case I:} $\ell =4m+1$, with $m\geqslant 0$ an integer. Then for all
integers $n\geqslant 0$, from Formula (\ref{eq2}), we deduce%
\begin{eqnarray*}
(T_{0}u_{1})(n+m+1) &=&u(4(n+m+1)+1)=u(\ell +4n+4)=\delta , \\
(T_{1}u_{1})(n+m) &=&u(4n+4m+3)=u(\ell +4n+2)=\beta .
\end{eqnarray*}%
Since both $T_{0}u_{1}$ and $T_{1}u_{1}$ are ultimately constant, then $%
u_{1} $ is ultimately periodic, and thus $2$-automatic.

Similarly, for all integers $n\geqslant 0$, we also have%
\begin{eqnarray*}
(T_{0}u_{0})(n+m+1) &=&u(4(n+m+1))=u(\ell +4n+3) \\
&=&\gamma (u(2n+2))^{r}=\gamma (u_{0}(n+1))^{r}, \\
(T_{1}u_{0})(n+m) &=&u(4n+4m+2)=u(\ell +4n+1) \\
&=&\alpha (u(2n+1))^{r}=\alpha (u_{1}(n))^{r}.
\end{eqnarray*}%
So $T_{1}u_{0}$ is ultimately periodic as $u_1$, and $(T_{0}u_{0})(n+m)=\sigma_{\gamma}(u_{0}(n))$ for all integers $n\geqslant 1$.
Then by Theorem \ref{thm3}, we
obtain that $u_{0}$ is $2$-automatic.
\smallskip

\textbf{Case II:} $\ell =4m+2$, with $m\geqslant 0$ an integer. Then for all
integers $n\geqslant 0$, from Formula (\ref{eq2}), we deduce%
\begin{eqnarray*}
(T_{0}u_{0})(n+m+1) &=&u(4(n+m+1))=u(\ell +4n+2)=\beta , \\
(T_{1}u_{0})(n+m+1) &=&u(4n+4m+6)=u(\ell +4n+4)=\delta .
\end{eqnarray*}%
So $u_{0}$ is ultimately periodic, and thus $2$-automatic.

Similarly, for all integers $n\geqslant 0$, we also have%
\begin{eqnarray*}
(T_{0}u_{1})(n+m+1) &=&u(\ell +4n+3)=\gamma (u(2n+2))^{r}=\gamma
(u_{0}(n+1))^{r}, \\
(T_{1}u_{1})(n+m) &=&u(\ell +4n+1)=\alpha (u(2n+1))^{r}=\alpha
(u_{1}(n))^{r}.
\end{eqnarray*}%
Hence $T_{0}u_{1}$ is ultimately periodic as $u_0$, and $(T_{1}u_{1})(n+m)=\sigma_\alpha
(u_{1}(n))$ for all integers $n\geqslant 1$. Then by Theorem \ref{thm3}, we
obtain that $u_{1}$ is $2$-automatic.
\smallskip

\textbf{Case III:} $\ell =4m+3$, with $m\geqslant 0$ an integer. Then for
all integers $n\geqslant 0$, from Formula (\ref{eq2}), we deduce%
\begin{eqnarray*}
(T_{0}u_{1})(n+m+1) &=&u(4(n+m+1)+1)=u(\ell +4n+2)=\beta , \\
(T_{1}u_{1})(n+m+1) &=&u(4n+4m+7)=u(\ell +4n+4)=\delta .
\end{eqnarray*}%
So $u_{1}$ is ultimately periodic, and thus $2$-automatic.

Similarly, for all integers $n\geqslant 0$, we also have%
\begin{eqnarray*}
(T_{0}u_{0})(n+m+1) &=&u(4(n+m+1))=u(\ell +4n+1) \\
&=&\alpha (u(2n+1))^{r}=\alpha (u_{1}(n))^{r}, \\
(T_{1}u_{0})(n+m+1) &=&u(4n+4m+6)=u(\ell +4n+3) \\
&=&\gamma (u(2n+2))^{r}=\gamma (u_{0}(n+1))^{r}.
\end{eqnarray*}%
Thus $T_{0}u_{0}$ is ultimately periodic as $u_1$, and $(T_{1}u_{0})(n+m)=\sigma_\gamma
(u_{0}(n))$ for all integers $n\geqslant 1$. Then by Theorem \ref{thm3}, we
obtain that $u_{0}$ is $2$-automatic.
\smallskip

\textbf{Case IV:} $\ell =4m+4$, with $m\geqslant 0$ an integer. Then for all
integers $n\geqslant 0$, from Formula (\ref{eq2}), we deduce%
\begin{eqnarray*}
(T_{0}u_{0})(n+m+2) &=&u(4(n+m+2))=u(\ell +4n+4)=\delta , \\
(T_{1}u_{0})(n+m+1) &=&u(4n+4m+6)=u(\ell +4n+2)=\beta .
\end{eqnarray*}%
So $u_{0}$ is ultimately periodic, and thus $2$-automatic.

Similarly, for all integers $n\geqslant 0$, we also have%
\begin{eqnarray*}
(T_{0}u_{1})(n+m+1) &=&u(\ell +4n+1)=\alpha (u(2n+1))^{r}=\alpha
(u_{1}(n))^{r}, \\
(T_{1}u_{1})(n+m+1) &=&u(\ell +4n+3)=\gamma (u(2n+2))^{r}=\gamma
(u_{0}(n+1))^{r}.
\end{eqnarray*}%
Hence $T_{1}u_{1}$ is ultimately periodic as $u_0$, and $(T_{0}u_{1})(n+m+1)=\sigma_\alpha
(u_{1}(n))$ for all integers $n\geqslant 1$. Then by Theorem \ref{thm3}, we
obtain that $u_{1}$ is $2$-automatic.
\end{proof}

\section{A substitutive but not automatic sequence}

In this section we are concerned with the following question: is it a
specificity of (certain) hyperquadratic continued fractions to generate, in
the way that we have described above, an automatic sequence?
To such a wide question, we will only give a very partial answer, by considering a last example. As we remarked in the introduction, the possibility of describing explicitly the
continued fraction expansion for an algebraic power series, which is not
hyperquadratic, appears to be remote. However, a particular example, which
was introduced by chance in \cite{MR}, does exist. This example is the object of the theorem below.
\medskip

First, we recall notions on substitutive sequences (see for example \cite{AS}).
\medskip

Let $A$ be an alphabet with $A=\lbrace a_1,a_2,...,a_N\rbrace$. A substitution on $A$ is a morphism $\sigma :$ $%
A^{\ast }\rightarrow A^{\ast }$. With the morphism $\sigma $, there is
associated a matrix $M_{\sigma }=(m_{i,j})_{1\leqslant i,j\leqslant N}$, where $m_{i,j}$ is
the number of occurrences of $a_i$ in the word $\sigma (a_j)$. Since $M_{\sigma
} $ is a non-negative integer square matrix, by the famous Frobenius-Perron theorem
(see for example \cite{G}), $M_{\sigma }$ has a real eigenvalue $\alpha $,
called the dominating eigenvalue of $M_{\sigma }$, which is an algebraic
integer and greater than or equal to the modulus of any other eigenvalue,
thus a Perron number. If there exists a letter $a\in A$ such that $\sigma
(a)=ax$ for some $x\in A^{\ast }\setminus \{\emptyset \}$, and $%
\lim_{n\rightarrow \infty }|\sigma ^{n}(a)|=+\infty $, then $\sigma $ is
said to be prolongable on $a$. Since for all integers $n\geqslant 0$, $%
\sigma ^{n}(a)$ is a prefix of $\sigma ^{n+1}(a)$, and $|\sigma ^{n}(a)|$
tends to infinity with $n$, the sequence $(\sigma ^{n}(a))_{n\geqslant 0}$
converges, and we denote its limits by $\sigma ^{\infty }(a)$. The latter is
the unique infinite fixed point of $\sigma $ starting with $a$. Let $o$ be a
mapping defined on $A$, extended pointwisely over $A^{\ast }\cup A^{\mathbb{N%
}}$. We put $v=o(\sigma ^{\infty }(a))$, and call it an $\alpha $-substitutive
sequence.

We have the following important characterization for automatic sequences 
in terms of substitutive sequences, due to  Cobham \cite{C1}: 
a sequence $v=(v(n))_{n\geqslant 1}$ is $k$-automatic if and only if $v$ is a substitutive sequence
with $\sigma$ such that $|\sigma (c)|=k$, for all $%
c\in A$. Note that in this case $v$ is $k$-substitutive.

Now let $\alpha ,\beta $ be two multiplicatively independent Perron numbers.
By generalizing another classical theorem of Cobham \cite{C2}, Durand has finally
shown in \cite[Theorem 1, p.1801]{D} the remarkable result that a sequence
is both $\alpha $-substitutive and $\beta $-substitutive if and only if it
is ultimately periodic.
\smallskip

We can now state and prove the following theorem.

\begin{theorem}
The algebraic equation $X^{4}+X^{2}-TX+1=0$ has a unique root $\alpha $ in $%
\mathbb{F}(3)$. Let $\alpha =[0,a_{1},a_{2},\ldots ,a_{n},\ldots ]$ be its
continued fraction expansion and $u(n)$ be the leading coefficient of $a_{n}$
for all integers $n\geqslant 1$. The sequence $W=(u(n))_{n\geqslant 1}$ is the limit of the sequence $(W_{n})_{n\geqslant 0}$ of finite
words over the alphabet $\{{1,2\}}$, defined recursively as follows:
\begin{equation}
W_{0}=\emptyset ,\text{ }W_{1}=1,\text{ and }%
W_{n}=W_{n-1},2,W_{n-2},2,W_{n-1},\text{ for all integers }n\geqslant 2,
\label{eq3}
\end{equation}%
where commas indicate here concatenation of words.
Then $\alpha $ is not hyperquadratic, and
the sequence $W=(u(n))_{n\geqslant 1}$ is substitutive but not automatic.
\end{theorem}

\begin{proof}
The existence in $\mathbb{F}(p)$ of the root of the quartic equation stated
in this theorem was observed firstly by Mills and Robbins in \cite{MR}, for
all prime numbers $p$. For $p=3$, in the same work, a conjecture on its
continued fraction expansion, based on computer observation, was given. Buck
and Robbins established this conjecture in \cite{BR}. Shortly after another
proof of this conjecture was given in \cite{L1}. We have $\alpha
=[0,a_{1},a_{2},\ldots ,a_{n},\ldots ]$ and the sequence of polynomials $%
(a_{n})_{n\geqslant 1}$ is obtained as the limit of a sequence of finite
words $(\Omega _{n})_{n\geqslant 0}$ with letters in $\mathbb{F}_{3}[T]$,
defined by:
\begin{equation}
\Omega _{0}=\emptyset ,\ \Omega _{1}=T,\ \text{and}\ \Omega _{n}=\Omega
_{n-1},2T,\Omega _{n-2}^{(3)},2T,\Omega _{n-1},  \label{eq4}
\end{equation}%
where commas indicate concatenation of words, and $\Omega _{n-2}^{(3)}$
denote the word obtained by cubing each letter of $\Omega _{n-2}$. Since $%
x^{3}=x$ for all $x$ in $\mathbb{F}_{3}^{\ast }$, we obtain immediately for $W$ the
desired formulas (\ref{eq3}) from (\ref{eq4}).

The fact that $\alpha $ is not hyperquadratic was proved in \cite{L1} (see
the remark after Theorem A, p.209). Indeed the knowledge of the continued
fraction allows to show that the irrationality measure is equal to $2$.
However the sequence of partial quotients is clearly unbounded, and it was
proved by Voloch \cite{V} that if $\alpha $ were hyperquadratic with an
unbounded sequence of partial quotients, then the irrationality measure
would be strictly greater than $2$ (the reader may consult \cite[p.215-216]%
{L2}, for a presentation of these general statements).
\smallskip

Now we show that $W$ is $(1+\sqrt{2})$-substitutive, but not automatic.
\smallskip

Put $A=\{a,b,c\}$, and define
\begin{equation*}
\sigma (a)=abca,\ \sigma (b)=ca,\ \sigma (c)=c,\ o(a)=1,\ o(b)=o(c)=2.
\end{equation*}

For all integers $n\geqslant 0$, set $V_{n}=\sigma ^{n}(a)$. Then for all
integers $n\geqslant 2$, we have%
\begin{equation*}
V_{n}=\sigma ^{n}(a)=\sigma ^{n-1}(abca)=V_{n-1}\sigma
^{n-2}(ca)cV_{n-1}=V_{n-1}cV_{n-2}cV_{n-1}.
\end{equation*}%
But we also have $W_{1}=o(a)$, and $W_{2}=1221=o(abca)=o(\sigma
(a))=o(V_{1}) $, thus $o(\sigma ^{n}(a))$ satisfies the same relations
as $W_{n+1}$, consequently they coincide. Set $%
v=\lim_{n\rightarrow \infty }\sigma ^{n}(a)$. Then $\sigma (v)=v$, and $%
W=o(v)$. So $W$ is substitutive. Finally we also have
\begin{equation*}
M_{\sigma }=\left(
\begin{tabular}{ccc}
2 & 1 & 0 \\
1 & 0 & 0 \\
1 & 1 & 1%
\end{tabular}%
\right) ,
\end{equation*}%
whose characteristic polynomial is equal to $(\lambda -1)(\lambda
^{2}-2\lambda -1)$, and $1+\sqrt{2}$ is the dominating eigenvalue.
Hence $W$ is $(1+\sqrt{2})$-substitutive.
Since $1+\sqrt{2}$ is multiplicatively independent
with all integers $k\geqslant 2$, according to Cobham's
characterization and Durand's theorem, we see that $W$
cannot be $k$-automatic unless it is ultimately periodic.
To conclude the proof, we need only prove that $W$ is not ultimately periodic.
To do so, we compute the frequency of $2$ in $W$.
For all integers $n\geqslant 0$, put $l_{n}=|W_{n}|$. Then we have
\begin{equation*}
l_{0}=0,\ l_{1}=1,\ \text{and}\ l_{n}=2l_{n-1}+l_{n-2}+2,\ \text{for all
integers}\ n\geqslant 2,
\end{equation*}%
from which we deduce $l_{n}=-1+\frac{2+\sqrt{2}}{4}(1+\sqrt{2})^{n}+%
\frac{2-\sqrt{2}}{4}(1-\sqrt{2})^{n}$, for all integers $n\geqslant 0$.
For all integers $n\geqslant 0$, let $m_{n}$ be the number of occurences of $%
2$ in $W_{n}$. Then
\begin{equation*}
m_{0}=m_{1}=0,\text{ and }m_{n}=2m_{n-1}+m_{n-2}+2,\text{ for all integers }%
n\geqslant 2,
\end{equation*}%
from which we obtain $m_{n}=-1+\frac{1}{2}\big((1+\sqrt{2})^{n}+((1-\sqrt{2})^{n}\big)$,
for all integers $n\geqslant 0$. If $W$ were ultimately periodic,
then the frequency of $2$ in $W$ would exist, and it would be a rational number,
in contradiction to $\lim_{n\rightarrow \infty }m_{n}/l_{n}=2-\sqrt{2}$.
\end{proof}

\noindent \textbf{Acknowledgments.} Jia-Yan Yao would like to thank the
National Natural Science Foundation of China (Grants no. 10990012 and
11371210) and the Morningside Center of Mathematics (CAS) for partial
financial support. 
\vskip 2 cm
\begin{tabular}{ll}
Alain LASJAUNIAS &  \\
Institut de Math\'{e}matiques de Bordeaux &  \\
CNRS-UMR 5251 &  \\
Talence 33405 &  \\
France &  \\
E-mail: Alain.Lasjaunias@math.u-bordeaux1.fr &  \\
&  \\
Jia-Yan YAO &  \\
Department of Mathematics &  \\
Tsinghua University &  \\
Beijing 100084 &  \\
People's Republic of China &  \\
E-mail: jyyao@math.tsinghua.edu.cn &
\end{tabular}


\smallskip

\begin{thebibliography}{99}
\bibitem{A} J.-P. Allouche, \emph{Sur le d\'{e}veloppement en fraction
continue de certaines s\'{e}ries formelles}, C. R. Acad. Sci. Paris \textbf{%
307} (1988), 631--633.

\bibitem{AS} J.-P. Allouche and J. Shallit, \emph{Automatic sequences.
Theory, applications, generalizations.} Cambridge University Press,
Cambridge (2003).

\bibitem{as2} J.-P. Allouche and J. Shallit, \emph{A variant of Hofstadter's
sequence and finite automata.} J. Aust. Math. Soc. \textbf{93} (2012), 1--8.

\bibitem{BS1} L. E. Baum and M. M. Sweet, \emph{Continued fractions of
algebraic power series in characteristic }$2$\emph{,} Ann. of Math. \textbf{%
103} (1976), 593--610.

\bibitem{BS2} L. E. Baum and M. M. Sweet, \emph{Badly approximable power
series in characteristic }$2$, Ann. of Math. \textbf{105} (1977), 573--580.

\bibitem{BR} M. Buck and D. Robbins, \emph{The continued fraction expansion
of an algebraic power series satisfying a quartic equation.} J. Number
Theory \textbf{50} (1995), 335--344.

\bibitem{C} G. Christol, \emph{Ensembles presques p\'{e}riodiques }$k$\emph{%
-reconnaissables.} Theorect. Comput. Sci. \textbf{9} (1979), 141-145.

\bibitem{CKMFR} G. Christol, T. Kamae, M. Mend\`{e}s France and G. Rauzy,
\emph{Suites alg\'{e}briques, automates et substitutions.} Bull. Soc. Math.
France \textbf{108} (1980), 401--419.

\bibitem{C2} A. Cobham, \emph{On the base-dependence of sets of numbers
recognizable by finite automata.} Math. Systems Theory \textbf{3} (1969),
186--192.

\bibitem{C1} A. Cobham, \emph{Uniform tag sequences.} Math. Systems Theory
\textbf{6} (1972), 164--192.

\bibitem{D} F. Durand, \emph{Cobham's theorem for substitutions.} J. Eur.
Math. Soc. (JEMS) \textbf{13} (2011), 1799--1814.

\bibitem{E} S. Eilenberg, \emph{Automata, Languages and Machines.} Vol. A.
Academic Press (1974).

\bibitem{F} A. Firicel, \emph{Sur le d\'{e}veloppement en fraction continue
d'une g\'{e}n\'{e}ralisation de la cubique de Baum et Sweet.} J. Th\'{e}or.
Nombres Bordeaux \textbf{22} (2010), 629--644.

\bibitem{G} F. R. Gantmacher, \emph{The theory of matrices.} Chelsea
Publishing Co., New York (1959).

\bibitem{L1} A. Lasjaunias, \emph{Diophantine approximation and continued
fraction expansions of algebraic power series in positive characteristic.}
J. Number Theory \textbf{65} (1997), 206--225.

\bibitem{L2} A. Lasjaunias, \emph{A survey of Diophantine approximation in
fields of power series.} Monatsh. Math. \textbf{130} (2000), 211--229.

\bibitem{L3} A. Lasjaunias, \emph{Continued fractions for hyperquadratic
power series over a finite field.} Finite Fields Appl. \textbf{14} (2008),
329--350.

\bibitem{LY} A. Lasjaunias and J.-Y. Yao, \emph{Hyperquadratic continued
fractions in odd characteristic with partial quotients of degree one.} J.
Number Theory \textbf{149} (2015), 259--284.

\bibitem{M} K. Mahler, \emph{On a theorem of Liouville in fields of positive
characteristic.} Canad. J. Math. \textbf{1} (1949), 397--400.


\bibitem{MR} W. Mills and D. P. Robbins, \emph{Continued fractions for
certain algebraic power series.} J. Number Theory \textbf{23} (1986),
388--404.

\bibitem{S} W. Schmidt, \emph{On continued fractions and diophantine
approximation in power series fields.} Acta Arith. \textbf{95} (2000),
139--166.


\bibitem{T} D. Thakur, \emph{Diophantine approximation exponents and
continued fractions for algebraic power series.} J. Number Theory \textbf{79}
(1999), 284--291.

\bibitem{V} J.-F. Voloch, \emph{Diophantine approximation in positive
characteristic.} Period. Math. Hungar. \textbf{19} (1988), 217--225.

\bigskip
\end{thebibliography}
\end{document}